\documentclass{amsart}
\usepackage[utf8]{inputenc}
\usepackage[all]{xy}
\usepackage{upgreek,amsthm,bbm,bm}
\usepackage{amssymb,amstext,amsmath,amscd,amsthm,hyperref,tikz,tikz-cd,amsfonts,enumerate,graphicx,latexsym,color}
\usepackage[all]{xy}
\newtheorem{thm}{Theorem}[section]
\newtheorem{cor}[thm]{Corollary}
\newtheorem{lem}[thm]{Lemma}
\newtheorem{prop}[thm]{Proposition}

\newtheorem*{thm*}{Theorem}
\newtheorem*{prop*}{Proposition}
\newtheorem*{cor*}{Corollary}
\theoremstyle{definition}

\newtheorem{defn}[thm]{Definition}
\newtheorem{rem}[thm]{Remark}

\newtheorem*{conj*}{Conjecture}
\newtheorem{exam}[thm]{Example}

\newtheorem*{claim*}{Claim}
\newtheorem*{ques*}{Question}

\newtheorem{chunk}[thm]{\hspace*{-1.065ex}\bf}
\theoremstyle{remark}

\numberwithin{equation}{thm}
\def\amp{\operatorname{\mathsf{amp}}}

\def\D{\mathcal{D}}
\def\depth{\operatorname{depth}}

\def\Ext{\operatorname{Ext}}

\def\ge{\geqslant}

\def\G{\operatorname{G}}

\def\Gpd{\operatorname{\mathsf{Gpd}}}
\def\Gd{\operatorname{\mathsf{G-dim}}}

\def\GCpd{\operatorname{\mathsf{G_C-pd}}}
\def\GCd{\operatorname{\mathsf{G_C-dim}}}

\def\GC{\operatorname{\mathsf{G_C}}}

\def\gr{\operatorname{\mathsf{grade}}}

\def\H{\operatorname{H}}

\def\Hom{\operatorname{Hom}}

\def\id{\mathrm{id}}

\def\Cid{\mathrm{C\text{-}id}}

\def\Ker{\operatorname{Ker}}

\def\pd{\operatorname{pd}}

\def\r{\operatorname{r}}
\def\rank{\operatorname{rank}}

\def\Spec{\operatorname{Spec}}
\def\Supp{\operatorname{Supp}}

\def\Tor{\operatorname{Tor}}

\newcommand{\uhom}{{\mathbf R}\Hom}
\newcommand{\utp}{\otimes^{\mathbf L}}

\newcommand{\fm}{\mathfrak{m}}
\newcommand{\fp}{\mathfrak{p}}

\newcommand{\vdim}{\mbox{vdim}\,}
\begin{document}

\title{Numerical aspects of complexes of finite homological dimensions}
\author[M. Rahro Zargar and Mohsen Gheibi]{Majid Rahro Zargar and Mohsen Gheibi}
\address{Majid Rahro Zargar, Department of Engineering Sciences, Faculty of Advanced Technologies, University of Mohaghegh Ardabili, Namin, Ardabil, Iran,}
\email{zargar9077@gmail.com}
\email{m.zargar@uma.ac.ir}
\address[M. Gheibi]{Department of Mathematics, Florida A{\&}M University, Tallahassee FL, USA }
\email{mohsen.gheibi@famu.edu}

\begin{abstract}
Let $(R,\fm)$ be a local ring, and let $C$ be a semidualizing complex. We establish the equality $r_R(Z) = \nu(\Ext^{g-\inf C}_R(Z,C))\mu^{\depth C}_R(\mathfrak{m}, C)$ for a homologically finite and bounded complex $Z$ with finite $\GC$-dimension $g$. Additionally, we prove that if $\Ext^i(M,N)=0$ for sufficiently large $i$, while $\id_R\Ext^i(M,N)$ remains finite for all $i$, then both $\pd_R M$ and $\id_R N$ are finite when $M$ and $N$ are finitely generated $R$-modules. These findings extend the recent results of Ghosh and Puthenpurakal \cite{Ghosh}, addressing their questions as presented in \cite[Question 3.9]{Ghosh} and \cite[Question 4.2]{Ghosh}.
\end{abstract}
\maketitle

%\tableofcontents

\section{Introduction}

Let $R$ be a commutative Noetherian local ring. Recently in \cite[Theorem 3.5]{Ghosh}, the authors proved the following theorem.

\begin{thm*}[Ghosh and Puthenpurakal]
    If $M$ is a finitely generated $\G$-perfect $R$-module with $\Gd_R M=g$, then
    $$r_R(M)=\nu(\Ext^{g}_R(M,R))r_R(R).$$
\end{thm*}
\noindent Here $r_R(-)$ and $\nu(-)$ represent the type and number of minimal generators,  respectively. In light of the aforementioned theorem, in this paper we present the following result which not only generalizes Ghosh and Puthenpurakal's result but also relaxes the $\G$-perfect assumption and addresses their question \cite[Question 3.9]{Ghosh}.

\begin{thm*}[A]
Let $(R,\fm)$ be a local ring and $C$ be a semidualizing $R$-complex. Let $Z\in\mathrm{D}_{\Box}^f(R)$ be an $R$-complex of finite $\GC$-dimension $g$. Then one has an  equality $$r_R(Z)=\nu(\Ext_R^{g-\inf C}(Z, C))\mu_{R}^{\depth_R(C)}(\fm,C).$$
\end{thm*}

\noindent As an application of Theorem (A), we provide some criteria for a semidualizing complex $C$ to be dualizing in terms of the existence of certain complexes of finite $\GC$-dimension; see Corollary \ref{C31}.

Next, in \cite[Theorem 4.3]{Ghosh}, the authors proved the following.

\begin{thm*}[Ghosh and Puthenpurakal]
    Let $N$ be a nonzero $R$-module of finite injective dimension with $\depth_RN\ge \dim R -1$. If $\id_R\Ext^i_R(M,N)<\infty$ for all $i$, then $\pd_RM<\infty$.
\end{thm*}

\noindent In addition, the authors raised the question of whether the aforementioned theorem holds true without the additional assumption $\depth_RN\ge \dim R -1$; see \cite[Question 4.2]{Ghosh}. In Section 4, we provide an affirmative answer to this question and present the following more comprehensive result.

\begin{prop*}[B]
Let $X, Y$ be complexes such that $X\in \mathrm{D}_{\sqsubset}^f(R)$ and $Y\in \mathrm{D}_{\sqsupset}^f(R)$. Assume $\Ext^i(X,Y)=0$ for all $i\gg 0$ and $\id_R\Ext^i(X,Y)<\infty$ for all $i$. Then, both $\pd_R X$ and $\id_R Y$ are finite.
\end{prop*}

In a related topic, in \cite{GTA}, Ghosh and Takahashi examined the celebrated Auslander-Reiten Conjecture in terms of injective dimensions of $\Hom_R(M,R)$ and $\Hom_R(M,M)$ for a finitely generated $R$-module $M$. In their work, they established the following result.

\begin{thm*}[Ghosh and Takahashi]
Let $R$ be a commutative Noetherian ring and let $M$ be a nonzero $R$-module such that $\Ext^i_R(M,M\oplus R)=0$ for all $i>0$. If at least one of $\Hom_R(M,R)$ or $\Hom_R(M,M)$ has finite injective dimension, then $M$ is free and $R$ is Gorenstein.
\end{thm*}

In Section 4, we provide a concise proof of this theorem, which can be found in Theorem \ref{AR}. Additionally, we extend Peskine-Szpiro's theorem to the realm of $C$-injective dimension.

\begin{thm*}[C]
    Let $(R,\fm)$ be a commutative Noetherian local ring, $C$ be a semidualizing $R$-complex, and $X\in\mathcal{D}_{\Box}^f(R)$. Assume that $\id_R(C\utp_RX)<\infty$. Then, for all integers $t$, there is an equality $$\beta^{R}_{t}(X)=\underset{\tiny{i+j=t}}\sum\mu_{R}^i(\fm,C)\mu_{R}^{-j}(\fm,C\utp_R X).$$
\end{thm*}

As a corollary of Theorem (C), one has a semidualizing module $C$ is dualizing, if there exists a cyclic $R$-module of finite $C$-injective dimension.

\section{Prerequisites}
Throughout, $R$ is a commutative Noetherian ring with identity.

\begin{chunk}[\bf{Complexes}]
 An $R$ complex $X$ is a sequence of $R$-modules $X_i$ and $R$-linear
maps $\partial_{i}^{X}: X_i\rightarrow X_{i-1}, i\in\mathbb{Z}$, where $X_i$ is an $R$-module in degree $i$, and
$\partial_{i}^{X}$ is the $i$-th \textit{differential} with $\partial_{i-1}^{X}\partial_{i}^{X}=0$.An $R$-module $M$ is thought of as a complex $0\rightarrow M\rightarrow0$, with $M$
in degree $0$. For any integer $n$, the $n$-fold shift of a complex $(X,\xi^X)$ is the complex $\Sigma^nX$ given by $(\Sigma^n X)_v=X_{v-n}$ and $\xi_{v}^{\Sigma^nX}=(-1)^n\xi_{v-n}^{X}$. The homological position and size of a complex are captured by the numbers
\textit{supremum}, \textit{infimum} and \textit{amplitude}
defined by $$\sup X:=\sup \{i\in \mathbb{Z}~|~\H_i(X)\neq 0\},$$ $$\inf X:=\inf \{i\in \mathbb{Z}~|~
\H_i(X) \neq 0\}, \text{and}$$
$$\amp X:=\sup X-\inf X.$$
 With the usual convention that $\sup \emptyset=-\infty$ and $\inf \emptyset=\infty$.
 \end{chunk}

\begin{chunk}[\bf{Derived Category}]
The \textit{derived category} $\mathrm{D}(R)$ is
the category of $R$-complexes localized at the class of all quasi-isomorphisms (see
\cite{RHart}). We use $\simeq$ to denote isomorphisms in $\mathrm{D}(R)$. The full subcategory of homologically left (resp. right) bounded complexes is denoted by $\mathrm{D}_{\sqsubset}(R)$ (resp.
$\mathrm{D}_{\sqsupset}(R)$). Also, we denote $\mathrm{D}_{\Box}^f(R)$ the full subcategory of homologically bounded complexes with finitely generated homology modules.

Let $\textbf{X},\textbf{Y}\in \mathrm{D}(R)$. The left-derived tensor product of $\textbf{X}$ and $\textbf{Y}$ in $\mathrm{D}(R)$ is denoted  $\textbf{X}\otimes_
R^{{\bf L}}\textbf{Y}$ and defined by $$\textbf{X}\otimes_R^{{\bf L}}\textbf{Y}\simeq \textbf{F}\otimes_R
\textbf{Y}\simeq \textbf{X} \otimes_R \textbf{F}^{'}\simeq \textbf{F}\otimes_R\textbf{F}^{'},$$ where $\textbf{F}$
and $\textbf{F}^{'}$ are semi-flat resolutions of $\textbf{X}$ and $\textbf{Y}$, respectively. The right-derived homomorphism complex
of $\textbf{X}$ and $\textbf{Y}$ in $\mathrm{D}(R)$ is denoted  ${\bf R}\Hom_R(\textbf{X},\textbf{Y})$ and
defined by $${\bf R}\Hom_R(\textbf{X},\textbf{Y}) \simeq \Hom_R(\textbf{P},\textbf{Y})\simeq \Hom_R(\textbf{X},
\textbf{I})\simeq \Hom_R(\textbf{P},\textbf{I}) ,$$ where $\textbf{P}$ is a semi projective resolution
of $\textbf{X}$ and $\textbf{I}$ is a semi-injective resolution of $\textbf{Y}$. For
$i\in\mathbb{Z}$, we set $\Tor^R_i(X,Y)=\H_i(X\utp_R Y)$ and $\Ext^i_{R}(X,Y)=\H_{-i}(\uhom_{R}(X,Y ))$. For $R$-complexes $X$, $Y$ and an integer $n\in \mathbb{Z}$ there is the following identity of complexes:
$$\uhom_R(X,\Sigma^n Y) = \Sigma^n\uhom_R(X,Y), \text{\ and}$$
$$\uhom_R(\Sigma^n X,Y) = \Sigma^{-n}\uhom_R(X,Y).$$
\end{chunk}
\begin{chunk}[\bf{Numerical Invariants}]
 Let $(R,\fm,k)$ be a local ring with residue field $k$. The \textit{depth} of an $R$-complex $X$ is defined by
$\depth_R(X) =-\sup\uhom_{R}(k,X)$,
and the Krull dimension of $X$ is defined as follows:
 \[\begin{array}{rl}
\dim_R(X)&=\sup\{\dim R/\fp-\inf X_{\fp}~|~ \fp\in\Spec R\}\\
&=\sup\{\dim R/\fp-\inf X_{\fp}~|~ \fp\in\Supp_R X\},
\end{array}\]
where $\Supp_R X = \{\fp\in\Spec R~|~X_\fp \not\simeq 0\}=\bigcup_{i\in\mathbb{Z}}\Supp_R \H_i(X)$. Note that for
modules these notions agree with the standard ones. For an $R$-complex $X$ in $\mathrm{D}_{\sqsubset}^f(R)$ such that $X\not\simeq 0$, one always has the inequality $\depth_R(X)\leq\dim_R(X)$. A homologically bounded and finitely generated $R$-complex $X$ is said to be \textit{Cohen-Macaulay} whenever the equality $\depth_R(X)=\dim_R(X)$ holds. For an $R$-complex $X$ and integer $m\in\mathbb{Z}$, the $m$-th \textit{Bass number} of $X$ is defined by $$\mu_R^m(\fm,X)=\rank_{k}\H_{-m}(\uhom_R(k,X))=\rank_{k}\Ext^m_{R}(k,X).$$ If $X$ is homologically bounded, then the \textit{type} of $X$ is denoted by $\r_R(X)$ and defined by $\r_R(X)=\mu_R^{\depth_R(X)}(\fm,X)$. The $m$-th \textit{Betti number} of $X$ is defined by $$\beta_m^{R}(X)=\rank_{k}\H_{m}(k\utp_{R}X)=\rank_{k}\Tor^R_{m}(k,X).$$
\end{chunk}

\begin{defn} An $R$-complex $C$ is said to be \textit{semidualizing} for $R$ if and only
if $C\in\mathrm{D}_{\Box}^f(R)$ and the homothety morphism $\chi_{C}^{R}:R\longrightarrow\uhom_{R}(C,C)$ is an isomorphism. Notice that in the case where $C$ is an $R$-module, this coincides with the definition of a semidualizing module. An $R$-complex $\D\in\mathrm{D}_{\Box}^f(R)$ is said to be a \textit{dualizing} complex if and only if it is a semidualizing complex with finite injective dimension.
\end{defn}
%%%%%%%%%%%%%%%%%%%%%%%%%%%%%%%%%%%%%%%%%%%%%%%%%%%%%%%%%%%%%%%%%%%%%%%%%%%%%%%%%%%%%%%%%%%%%%%%%%%%%%%%%%%%%%%%%%%%%%%%%%%%%%%%%%%%%%%%%%%%%%%%%%%%%%%%%%%%%%%%%%%%
\begin{defn}\label{2.2} Let $C$ be a semidualizing complex for $R$. An $R$-complex $X$ in $\mathrm{D}(R)$ is said to be $C$-\textit{reflexive} if and only if $X$ and $\uhom_{R}(X,C)$ belong to $\mathrm{D}_{\Box}^f(R)$ and the biduality morphism $\delta_{X}^C:X\longrightarrow\uhom_R(\uhom_R(X,C),C)$ is invertible. The class of complexes in $\mathrm{D}_{\Box}^f(R)$ which are $C$-reflexive is denoted by $\mathcal{R}(C)$. In view of \cite[Definition 3.11]{CF1}, for an $R$-complex $X$ in $\mathrm{D}_{\Box}^f(R)$ we define the {\em $G$-dimension} of $X$ with respect to $C$ as follows:
\[ \GCd_{R}(X)=\begin{cases}
       \inf C-\inf\uhom_R(X,C)& \text{if $X\in\mathcal{R}(C)$}\\
       \infty& \text{if $X\notin\mathcal{R}(C)$}.
       \end{cases} \]

\end{defn}
\begin{defn}\label{2.3}Let $C$ be a semidualizing $R$-module. An $R$-module $M$ is said to be $C$-\textit{Gorenstein projective} if:
\begin{itemize}
\item[(i)]{{$\Ext^{i\ge 1}_{R}(M,C\otimes_R P)=0$ for all projective $R$-modules $P$.}}
\item[(ii)]{{There exist projective $R$-modules $P_0, P_1,\dots$ together with an exact sequence:
$$0\longrightarrow M \longrightarrow C\otimes_R P_0\longrightarrow C\otimes_R P_1\longrightarrow\dots,$$
and furthermore, this sequence stays exact when we apply the functor
$\Hom_R(-,C\otimes_{R}Q)$ for any projective $R$-module $Q$.}}
\end{itemize}For a homologically right-bounded $R$-complex $X$ we define:
$$\GCpd_R(X)=\inf_{Y}\{\sup\{n\in\mathbb{Z}| Y_n\neq 0\}\},$$
where the infimum is taken over all $C$-Gorenstein projective resolutions $Y$ of $X$. In view of \cite[Examples 2.8(b)(c)]{HOJ}, such resolutions always exist. Note that when $C=R$ this notion recovers the concept of
ordinary Gorenstein projective dimension.
\end{defn}

Next, we recall the concept of the Gorenstein dimension with respect to a semidualizing $R$-module $C$, which was originally introduced by Golod \cite{Golod}.

\begin{defn}\label{DEf}The {\em $G_{C}$-class} of $R$ denoted $G_{C}(R)$, is the collection of finite $R$-module $M$ such that
\begin{itemize}
\item[(i)]{{$\Ext^i_{R}(M,C)=0$ for all $i>0$.}}
\item[(ii)]{{$\Ext^i_{R}(\Hom_{R}(M,C),C)=0$ for all $i>0$.}}
\item[(iii)]{{The canonical map $M\longrightarrow\Hom_{R}(\Hom_{R}(M,C),C)$ is an isomorphism.}}
\end{itemize}
For a non-negative integer $n$, the $R$-module $M$ is said to be of
{\em $G_{C}$-dimension} at most $n$, if and only if there exists an exact
sequence
$$0\longrightarrow G_n\longrightarrow G_{n-1}\longrightarrow\cdots \longrightarrow G_0\longrightarrow M\longrightarrow0,$$
where {$G_{i}\in G_{C}(R)$} for $0\leq i\leq n$. If such a sequence does not exist, then we write {$\Gd_{C}(M)=\infty$}.
Note that when $C=R$ this notion recovers the concept of Gorenstein dimension which was introduced in \cite{Aus}.
\end{defn}
%%%%%%%%%%%%%%%%%%%%%%%%%%%%%%%%%%%%%%%%%%%%%%%%%%%%%%%%%%%%%%%%%%%%%%%%%%%%%%%%%%%%%%%%%%%%%%%%%%%%%%%%%%%%%%%%%%%%%%%%%%%%%%%%%%%%%%%%%%%%%%%%%%%%%%%%%%%%%%%%
\begin{rem}Notice that if $C$ is a semidualizing $R$-module and $X\in\mathcal{\D}_{\Box}^f(R)$, then, by \cite[Proposition 3.1]{HOJ}, the Definitions \ref{2.2} and \ref{2.3} are coincide, that is, $\GCd_{R}(X)=\GCpd_R (X)$. By, \cite[Theorem 2.16]{HOJ}, for a homologically
right-bounded complex $Z$ of $R$-modules there is the equality $\GCpd_R (Z)=\Gpd_{R\ltimes C}(Z)$. Also, by \cite[Theorem 2.7]{SY}, for a finitely generated $R$-module $M$ and a semidualizing $R$-module $C$, we have that the two Definitions \ref{2.2} and \ref{DEf} are coincide, that is, $\Gd_{C}(M)=\GCd_{R}(M)$.
\end{rem}
\begin{defn}Let $C$ be a semidualizing complex for $R$. The {\it{Auslander class}} of $R$ with respect to $C$, $\mathcal{A}_C(R)$, is the full subcategories of $\mathrm{D}_{\Box}(R)$ defined by specifying their objects as follows: $X$ belongs to $\mathcal{A}_C(R)$ if and only if $C\utp_R X\in\mathrm{D}_{\Box}(R)$ and the canonical map $\gamma_X^C: X\rightarrow\uhom_R(C,C\utp_R X)$ is an isomorphism.
\end{defn}
%%%%%%%%%%%%%%%%%%%%%%%%%%%%%%%%%%%%%%%%%%%%%%%%%%%%%%%%%%%%%%%%%%%%%%%%%%%%%%%%%%%%%%%%%%%%%%%%%%%%%%%%%%%%%%%%%%%%%%%%%%%%%%%%%%%%%%%%%%%%%%%%%%%%%%%%%%%%%%%%
%%%%%%%%%%%%%%%%%%%%%%%%%%%%%%%%%%%%%%%%%%%%%%%%%%%%%%%%%%%%%%%%%%%%%%%%%%%%%%%%%%%%%%%%%%%%%%%%%%%%%%%%%%%%%%%%%%%%%%%%%%%%%%%%%%%%%%%%%%%%%%%%%%%%%%%%%%%%%%%%

\section{Complexes with finite Gorenstein dimension}
%%%%%%%%%%%%%%%%%%%%%%%%%%%%%%%%%%%%%%%%%%%%%%%%%%%%%%%%%%%%%%%%%%%%%%%%%%%%%%%%%%%%%%%%%%%%%%%%%%%%%%%%%%%%%%%%%%%%%%%%%%%%%%%%%%%%%%%%%%%%%%%%%%%%%%%%%%%%%%%%%%%%
%%%%%%%%%%%%%%%%%%%%%%%%%%%%%%%%%%%%%%%%%%%%%%%%%%%%%%%%%%%%%%%%%%%%%%%%%%%%%%%%%%%%%%%%%%%%%%%%%%%%%%%%%%%%%%%%%%%%%%%%%%%%%%%%%%%%%%%%%%%%%%%%%%%%%%%%%%%%%%%%%%%%

In this section, we give the proof of our main result Theorem (A). The starting point of this section is the following lemma, which is used in the proof of the next theorem. For the proof of this lemma, the reader is referred to \cite{CF1}.
\begin{lem}{\label{Lem}}Let $(R,\fm,k)$ be a local ring, $C$ be a semidualizing $R$-complex and $X$ be a homologically bounded and finitely generated $R$-complex. Then the following statements hold:
\begin{itemize}
\item[(1)]{{If $\GCd_R(X)<\infty$, then $\Gd_C(X)=\depth R- \depth_R(X).$}}
\item[(2)]{The following conditions are equivalent:}
\begin{itemize}
\item[(i)]{{$C$ is dualizing.}}
\item[(ii)]{{$\GCd_R(k)<\infty$.}}
\end{itemize}
\item[(3)]{{If $\GCd_R(X)<\infty$, then $\inf\uhom_R(X,C)=\depth_R(X)-\depth_R(C).$}}
\end{itemize}
\end{lem}
%%%%%%%%%%%%%%%%%%%%%%%%%%%%%%%%%%%%%%%%%%%%%%%%%%%%%%%%%%%%%%%%%%%%%%%%%%%%%%%%%%%%%%%%%%%%%%%%%%%%%%%%%%%%%%%%%%%%%%%%%%%%%%%%%%%%%%%%%%%%%%%%%%%%%%%%%%%%%%%%%%%%%%%%%%%%%%%%%%%%%%%%%%%%%%
%%%%%%%%%%%%%%%%%%%%%%%%%%%%%%%%%%%%%%%%%%%%%%%%%%%%%%%%%%%%%%%%%%%%%%%%%%%%%%%%%%%%%%%%%%%%%%%%%%%%%%%%%%%%%%%%%%%%%%%%%%%%%%%%%%%%%%%%%%%%%%%%%%%%%%%%%%%%%%%%%%%%%%%%%%%%%%%%%%%%%%%%%%%%%%
\begin{defn}Let $C$ be a semidualizing complex and $X$ be a homologically finitely generated and bounded complex. Then we define garde of $X$ with respect to $C$ as follow: $$\gr_C (X):=\inf\{ i | \Ext_R^i(X,C)\neq 0\}=-\sup\uhom_R(X,C).$$ Also, we say that a finitely generated $R$-module $M$ with finite $G$-dimension is $G$-perfect if $\gr_R(M)=\Gd_R(M)$.
\end{defn}

%%%%%%%%%%%%%%%%%%%%%%%%%%%%%%%%%%%%%%%%%%%%%%%%%%%%%%%%%%%%%%%%%%%%%%%%%%%%%%%%%%%%%%%%%%%%%%%%%%%%%%%%%%%%%%%%%%%%%%%%%%%%%%%%%%%%%%%%%%%%%%%%%%%%%%%%%%%%%%%%%%%%%%%%%%%%%%%%%%%%%%%%%%%%%%
%%%%%%%%%%%%%%%%%%%%%%%%%%%%%%%%%%%%%%%%%%%%%%%%%%%%%%%%%%%%%%%%%%%%%%%%%%%%%%%%%%%%%%%%%%%%%%%%%%%%%%%%%%%%%%%%%%%%%%%%%%%%%%%%%%%%%%%%%%%%%%%%%%%%%%%%%%%%%%%%%%%%%%%%%%%%%%%%%%%%%%%%%%%%%%
In \cite[Theorem 3.5]{Ghosh}, the authors proved that for a $G$-perfect $R$-module $M$ the equality $r_R(M)=\nu(\Ext_R^{g}(M, R))r_R(R)$ holds true. Also, they provided some examples to show that the above equality holds true for some $R$-modules with finite $G$-dimension which is not necessarily $G$-perfect. In this direction, in \cite[Question 3.9]{Ghosh} they asked whether the aforementioned equality holds true for all $R$-modules with finite $G$-dimension. The following theorem, which is our main result, provides an affirmative answer to the above question.

\begin{thm}{\label{Th3}} Let $(R,\fm)$ be a local ring and $C$ be a semidualizing $R$-complex. Suppose that there exists an $R$-complex $Z\in\mathcal{D}_{\Box}^f(R)$ with $g:=\GCd_{R}(Z)<\infty$, then one has the following equality $$r_R(Z)=\nu(\Ext_R^{g-\inf C}(Z, C))\mu_{R}^{\depth_R(C)}(\fm,C).$$
In particular, if $M$ is a finitely generated $R$-module with finite Gorenstein dimension $g$, then $r_R(M)=\nu(\Ext_R^{g}(M, R))r_R(R)$.

\begin{proof} Let $Z$ be a finite $G_{C}$-dimension $R$-complex in $\mathcal{D}_{\Box}^f(R)$, and notice that by \cite[Lemma 3.12]{CF1} one has $\GCd_{R}(\Sigma^{-g}Z)=\GCd_{R}(Z)-g=0$. Set $d=\depth_R(Z)$. Then by Lemma \ref{Lem}(1) we have the following equalities
\[\begin{array}{rl}
\H_{-d}(\uhom_{R}(k,Z))&=\H_{-d}(\uhom_{R}(k,\Sigma^{g}(\Sigma^{-g}Z)))\\
&=\H_{-d}(\Sigma^g\uhom_{R}(k,\Sigma^{-g}Z))\\
&=\H_{-(d+g)}(\uhom_{R}(k,\Sigma^{-g}Z))\\
&=\H_{-\depth R}(\uhom_{R}(k,\Sigma^{-g}Z)).\\
\end{array}\]
Next, set $X:=\Sigma^{-g}Z$ and notice that $X^{\dag\dag}\simeq X$, where $(-)^\dag:=\uhom_R(-,C)$. Hence, the above equalities yields the following equalities
\[\begin{array}{rl}
r_R(Z)&=\vdim_k\H_{-\depth R}(\uhom_{R}(k,X))\\
&=\vdim_k\H_{-\depth R}(\uhom_{R}(k,X^{\dag\dag}))\\
&=\vdim_k\H_{-\depth R}(\uhom_{R}(k,\uhom_R(X^{\dag},C))\\
&=\vdim_k\H_{-\depth R}(\uhom_{R}(k\utp_R X^{\dag},C))\\
&=\vdim_k\H_{-\depth R}(\uhom_{R}(k\utp_R X^{\dag}\utp_k k,C))\\
&=\vdim_k\H_{-\depth R}(\uhom_{k}(k\utp_R X^{\dag},\uhom_R(k,C)))\\
&=\underset{\tiny{i+j=\depth R}}\sum\beta^{R}_i(X^{\dag})\mu_{R}^j(\fm,C).
\end{array}\]
On the other hand, in view of \cite[6.2.4, 6.2.9]{CHFox}, one has $\inf\{i\in\mathbb{Z} ~|~\beta_{i}^R(X^{\dag})\neq 0)\}=\inf X^{\dag}$ and $\inf\{i\in\mathbb{Z} ~|~\mu_{R}^i(\fm, C)\neq 0)\}=\depth_R(C)$. As $\GCd_{R}(X) =0$, by using Lemma \ref{Lem} and the fact that $\inf C+\depth_R(C)=\depth R$, one has $\inf X^{\dag}=\inf C$. Therefore, one can deduce that
$r_R(Z)=\beta^{R}_{\inf C}(X^{\dag})\mu_{R}^{\depth_R(C)}(\fm,C)$. 

It remains to show that $\beta^{R}_{\inf C}(X^{\dag})=\nu(\Ext_R^{g-\inf C}(Z, C))$. First notice that by Lemma \ref{Lem}, we have $\inf C -g=\inf C -\depth R+ \depth_R(Z)=\inf\uhom_{R}(Z,C)$. Now we have the following equalities
\[\begin{array}{rl}
\beta^{R}_{\inf C}(X^{\dag})&=\beta^{R}_{\inf C}(\uhom_{R}(\Sigma^{-g}Z,C))\\
&=\beta^{R}_{\inf C}(\Sigma^{g}\uhom_{R}(Z,C))\\
&=\vdim_k\H_{\inf C}(k\utp_R\Sigma^{g}\uhom_{R}(Z,C))\\
&=\vdim_k\H_{\inf C}(\Sigma^{g}(k\utp_R\uhom_{R}(Z,C))\\
&=\vdim_k\H_{\inf C -g}(k\utp_R\uhom_{R}(Z,C))\\
&=\vdim_k\H_{\inf Z^{\dag}}(k\utp_R\uhom_{R}(Z,C))\\
&\overset{\ddag}=\vdim_k(\H_0(k)\utp_R\H_{\inf Z^{\dag}}(\uhom_{R}(Z,C)))\\
&=\nu(\Ext_R^{-\inf Z^{\dag}}(Z, C))\\
&=\nu(\Ext_R^{g-\inf C}(Z, C)).
\end{array}\]
Note that the equality $\ddag$ follows from \cite[Lemma 4.3.8]{CHFox}.
\end{proof}
\end{thm}

%%%%%%%%%%%%%%%%%%%%%%%%%%%%%%%%%%%%%%%%%%%%%%%%%%%%%%%%%%%%%%%%%%%%%%%%%%%%%%%%%%%%%%%%%%%%%%%%%%%%%%%%%%%%%%%%%%%%%%%%%%%%%%%%%%%%%%%%%%%%%%%%%%%%%%%%%%%%%%%%%%%
%%%%%%%%%%%%%%%%%%%%%%%%%%%%%%%%%%%%%%%%%%%%%%%%%%%%%%%%%%%%%%%%%%%%%%%%%%%%%%%%%%%%%%%%%%%%%%%%%%%%%%%%%%%%%%%%%%%%%%%%%%%%%%%%%%%%%%%%%%%%%%%%%%%%%%%%%%%%%%%%%%%

As a consequence of Theorems \ref{Th3} we obtain the following characterizations of a dualizing complex $C$ in terms of the existence of certain complexes of finite $G_C$-dimension which also provides a suitable generalization of \cite[Corollary 3.10]{Ghosh}.

\begin{cor}\label{C31}Let $(R,\fm)$ be a local ring and $C$ be a semidualizing $R$-complex. Consider the following statements:
\begin{itemize}
\item[(i)]{$C$ is dualizing}.
\item[(ii)]{$R$ admits a Cohen-Macaulay complex $X$ with finite $G_{C}$-dimension $g$ such that $r_R(X)\leq\nu(\Ext_R^{g-\inf C}(X, C))$.}
\item[(iii)] {$R$ admits a Cohen-Macaulay complex $X$ with finite $G_{C}$-dimension  $g$ such that $\dim_R(X)=\dim_R(C)-\gr_C(X)$ and $r_R(X)\leq\nu(\Ext_R^{g-\inf C}(X, C))$.}
\end{itemize}
Then, the implications \emph{(i)$\Longleftrightarrow$(iii) and (i)$\Longrightarrow$(ii)} hold and the implication \emph{(ii)$\Longrightarrow$(i)} holds whenever $\amp X=0$.
\end{cor}

\begin{proof}Fist note that a dualizing complex $C$ is a type one Cohen-Macaulay $R$-complex with finite $G_{C}$-dimension and $\gr_C(C)=0$. Therefore, the implications (i)$\Longrightarrow$(ii) and (i)$\Longrightarrow$(iii) follow for $X=C$. For the implications (ii)$\Longrightarrow$(i), by Theorem \ref{Th3}, one has $r_R(C)=1$, and so it follows from \cite[Theorem 3.7]{MRZ}. For the (iii)$\Longrightarrow$(i), the assumption $r_R(X)\leq\nu(\Ext_R^{g-\inf C}(X, C))$ and Theorem \ref{Th3} implies that $r_R(C)=1$. Also, in view of \cite[Proposition 3.9]{MRZ}, one has $C$ is Cohen-Macaulay, and so by \cite[Theorem 3.3]{MRZ} the assertion holds.
\end{proof}

%%%%%%%%%%%%%%%%%%%%%%%%%%%%%%%%%%%%%%%%%%%%%%%%%%%%%%%%%%%%%%%%%%%%%%%%%%%%%%%%%%%%%%%%%%%%%%%%%%%%%%%%%%%%%%%%%%%%%%%%%%%%%%%%%%%%%%%%%%%%%%%%%%%%%%%%%%%%%%%%%%%
%%%%%%%%%%%%%%%%%%%%%%%%%%%%%%%%%%%%%%%%%%%%%%%%%%%%%%%%%%%%%%%%%%%%%%%%%%%%%%%%%%%%%%%%%%%%%%%%%%%%%%%%%%%%%%%%%%%%%%%%%%%%%%%%%%%%%%%%%%%%%%%%%%%%%%%%%%%%%%%%%%%
\section{Finite injective dimension}

In this section, we address a question of Ghosh and Puthenpurakal \cite[Question 4.2]{Ghosh}. Additionally, we present a concise proof of Ghosh and Takahashi's result \cite[Corollary 1.3]{GTA}. Furthermore, we extend a theorem by Peskine and Szpiro, which asserts that if a local ring $R$ possesses a cyclic module with finite injective dimension, then $R$ is Gorenstein.

\begin{lem}\label{GG}
Let $X:= \cdots \longrightarrow X_{l+1}\overset{\partial_{l+1}^X} \longrightarrow  X_{l} \overset{\partial_{l}^X}\longrightarrow  X_{l-1} \to \cdots$ be a bounded complex such that $\id_R \H_i(X)<\infty$ for all $i$. Then $\id_R X<\infty$.
\begin{proof}
Let $n$ be the number of non-zero homologies of $X$. We proceed by induction on $n$. Assume that $n=1$. Then without loss of generality, we may assume $\H_0(X)\neq 0$. Let $Z=\Ker(\partial_0)$. Then the complex $X$ is isomorphic to $Y:=\cdots \longrightarrow X_1 \longrightarrow Z  \longrightarrow 0$ and the latter is isomorphic to $\H_0(X)$ in $\mathrm{D}(R)$. Since $\id_R\H_0(X)$ is finite, so is $\id_RX$.

Let $n>1$. Again without loss of generality, we may assume $\inf X=0$, and same as the above we have $X \simeq Y$. Since $\H_0(X)=\H_0(Y)$ has a finite injective dimension, there exists a bounded injective resolution $I$ of $\H_0(X)$. Let $C:=\cdots \longrightarrow X_1  \longrightarrow Z  \longrightarrow I^0  \longrightarrow I^1  \longrightarrow \cdots$, where the map $Z \longrightarrow I^0$ is the composition of $Z \twoheadrightarrow \H_0(Y)$ and $\H_0(Y) \hookrightarrow I^0$. This gives an exact sequence of complexes $0 \longrightarrow Y  \longrightarrow C  \longrightarrow I \longrightarrow 0$, and also one has $\H_i(C)=0$ for $i\leq 0$. Therefore,  by inductive hypothesis $\id_R C<\infty$. On the other hand, we have that $\id_RI$ is finite, and so $\id_RY<\infty$ as well. Therefore $\id_RX<\infty$.
\end{proof}
\end{lem}
%%%%%%%%%%%%%%%%%%%%%%%%%%%%%%%%%%%%%%%%%%%%%%%%%%%%%%%%%%%%%%%%%%%%%%%%%%%%%%%%%%%%%%%%%%%%%%%%%%%%%%%%%%%%%%%%%%%%%%%%%%%%%%%%%%%%%%%%%%%%%%%%%%%%%%%%%%%%%%%%
%%%%%%%%%%%%%%%%%%%%%%%%%%%%%%%%%%%%%%%%%%%%%%%%%%%%%%%%%%%%%%%%%%%%%%%%%%%%%%%%%%%%%%%%%%%%%%%%%%%%%%%%%%%%%%%%%%%%%%%%%%%%%%%%%%%%%%%%%%%%%%%%%%%%%%%%%%%%%%%%

\begin{prop}\label{MM}Let $X, Y$ be complexes such that $X\in \mathrm{D}_{\sqsubset}^f(R)$ and $Y\in \mathrm{D}_{\sqsupset}^f(R)$. Assume $\Ext^i(X,Y)=0$ for $i\gg 0$ and $\id_R\Ext^i(X,Y)<\infty$ for all $i$. Then, both $\pd_R X$ and $\id_R Y$ are finite.

\begin{proof}First note that $\Ext^{i}(X,Y)=\H_{-i}(\uhom_R(X,Y)$ for all $i$ and also by \cite[Lemma 4.2.9]{CHFox} one has $\uhom_R(X,Y)\in\mathrm{D}_{\sqsubset}^f(R)$. Therefore, in view of the Lemma \ref{GG} and \cite[Lemma 6.2.10]{CHFox} one can deduce that $\pd_R X$ and $\id_R Y$ are finite, as required.
\end{proof}
\end{prop}

%%%%%%%%%%%%%%%%%%%%%%%%%%%%%%%%%%%%%%%%%%%%%%%%%%%%%%%%%%%%%%%%%%%%%%%%%%%%%%%%%%%%%%%%%%%%%%%%%%%%%%%%%%%%%%%%%%%%%%%%%%%%%%%%%%%%%%%%%%%%%%%%%%%%%%%%%%%%%%%%%%%
%%%%%%%%%%%%%%%%%%%%%%%%%%%%%%%%%%%%%%%%%%%%%%%%%%%%%%%%%%%%%%%%%%%%%%%%%%%%%%%%%%%%%%%%%%%%%%%%%%%%%%%%%%%%%%%%%%%%%%%%%%%%%%%%%%%%%%%%%%%%%%%%%%%%%%%%%%%%%%%%%%%
In \cite[Proposition 4.1]{Ghosh}, the authors proved that over a Cohen-Macaulay local ring $R$ admitting a canonical module $\omega_R$, if $M$ is finitely generated $R$-module such that $\id_R\Ext^i(M,\omega_R)<\infty$ for all $i$, then $\pd_R M$ is finite. In this regard, in \cite[Question 4.2]{Ghosh}, they asked whether one can replace $\omega_R$ with any arbitrary $R$-module of finite injective dimension. The following corollary affirmatively answers their question even without assuming the finiteness of the injective dimension of the $R$-module $N$. Also, this result provides a partial answer to \cite[Question 2.18]{GTA}

\begin{cor}Let $(R,\fm)$ be a local ring and suppose that $M, N$ are finitely generated $R$-modules such that $\Ext^i(M,N)=0$ for $i\gg 0$, and $\id_R\Ext^i(M,N)<\infty$ for all $i$. Then, $\pd_R M$ and $\id_R N$ both are finite.
In particular, if $\Ext^i(M,M)=0$ for $i\gg 0$ and $\id_R\Ext^i(M,M)<\infty$ for all $i$, then $R$ is Gorenstein.
\end{cor}
Notice that the equality $\dim(R)-\depth(M)=\sup\{~i~|~\Ext^i_R(M,M)\neq 0\}$ holds for all finitely generated $R$-modules $M$ with finite injective dimension. Therefore, the following corollary which is an immediate consequence of the Proposition \ref{MM}, provides an affirmative answer to  \cite[Question 4.7]{Ghosh}.
\begin{cor}\label{C3}Let $M$ be a nonzero $R$-module such that the injective dimensions of
$M$ and $\Ext^i_R(M,M)$ are finite for all $0 \leq i \leq \dim(R)-\depth(M)$. Then $R$ is Gorenstein.
\end{cor}

%%%%%%%%%%%%%%%%%%%%%%%%%%%%%%%%%%%%%%%%%%%%%%%%%%%%%%%%%%%%%%%%%%%%%%%%%%%%%%%%%%%%%%%%%%%%%%%%%%%%%%%%%%%%%%%%%%%%%%%%%%%%%%%%%%%%%%%%%%%%%%%%%%%%%%%%%%%%%%
%%%%%%%%%%%%%%%%%%%%%%%%%%%%%%%%%%%%%%%%%%%%%%%%%%%%%%%%%%%%%%%%%%%%%%%%%%%%%%%%%%%%%%%%%%%%%%%%%%%%%%%%%%%%%%%%%%%%%%%%%%%%%%%%%%%%%%%%%%%%%%%%%%%%%%%%%%%%%%%
The following example shows that the finiteness of injective dimensions of Ext modules in Proposition \ref{MM} is necessary.
\begin{exam}Let $(R,\fm, k)$ be a non regular Gorenstein local ring of dimension $d$. Then, one has $\Ext^{i}_R (k,R)=0$ for all $i\neq d$ and $\Ext^{d}_R (k,R)\cong k$
Therefore, $\id_R(\Ext^d_{R}(k,R))=\infty$ and also $\pd_R(k)=\infty$.
\end{exam}
%%%%%%%%%%%%%%%%%%%%%%%%%%%%%%%%%%%%%%%%%%%%%%%%%%%%%%%%%%%%%%%%%%%%%%%%%%%%%%%%%%%%%%%%%%%%%%%%%%%%%%%%%%%%%%%%%%%%%%%%%%%%%%%%%%%%%%%%%%%%%%%%%%%%%%%%%%%%%
%%%%%%%%%%%%%%%%%%%%%%%%%%%%%%%%%%%%%%%%%%%%%%%%%%%%%%%%%%%%%%%%%%%%%%%%%%%%%%%%%%%%%%%%%%%%%%%%%%%%%%%%%%%%%%%%%%%%%%%%%%%%%%%%%%%%%%%%%%%%%%%%%%%%%%%%%%%%%
\begin{thm}\label{AR} The Auslander-Reiten conjecture holds true for a nonzero finitely generated $R$-module $M$ over a commutative Noetherian ring $R$ when at least one of $\Hom_R(M,R)$ or $\Hom_R(M,M)$ has finite injective dimension.
\begin{proof}By the assumptions of Auslander-Reiten conjecture, we have $\Ext^i_R(M,R)=0$ and $\Ext^i_R(M,M)=0$ for all $i\geq 1$. Therefore, in view of Proposition \ref{MM} and Corollary \ref{C3}, one has $\pd_R M$ and $\id_R R$ both are finite. Note that $\uhom_R(M,R)\simeq\Hom_R(M,R)$ in $\mathrm{D}(R)$. Now, consider the following equalities
\[\begin{array}{rl}
\mu_{R}^{t}(\fm, \Hom_R(M,R))&=\vdim_k\H_{-t}(\uhom_R(k, \Hom_R(M,R)))\\
&=\vdim_k\H_{-t}(\uhom_{R}(k, \uhom_{R}(M,R)))\\
&=\vdim_k\H_{-t}(\uhom_{R}(k\utp_{R}M,R))\\
&=\vdim_k\H_{-t}(\uhom_{R}(M\utp_{R}k\utp_{k}k,R))\\
&=\vdim_k\H_{-t}(\uhom_{k}(M\utp_{R}k,\uhom_{R}(k,R)))\\
&=\underset{\tiny{i\in\mathbb{Z}}}\sum\vdim_k\H_{i}(M\utp_{R}k)\vdim_k\H_{i-t}(\uhom_{R}(k,R))\\
&=\underset{\tiny{i+j=t}}\sum\beta^{R}_i(M)\mu_{R}^{j}(\fm,R).
\end{array}\]
Set $\dim R=d$ and note that $\mu_{R}^{d+1}(\fm, \Hom_R(M,R))=0$ and $\mu_{R}^{d+1}(\fm, R)=0$ as injective dimension of the $R$-modules $\Hom_R(M,R)$ and $R$ are finite. Therefore, by the above equality one has $$\beta^{R}_1(M)\mu_{R}^{d}(\fm,R)+\beta^{R}_2(M)\mu_{R}^{d-1}(\fm,R)+\cdots+\beta^{R}_{d+1}(M)\mu_{R}^{0}(\fm,R)=0.$$
Since $\mu_{R}^{d}(\fm,R)\neq 0$, we have $\beta^{R}_1(M)=0$, and therefore $M$ is a free $R$-module as required.
\end{proof}
\end{thm}
%%%%%%%%%%%%%%%%%%%%%%%%%%%%%%%%%%%%%%%%%%%%%%%%%%%%%%%%%%%%%%%%%%%%%%%%%%%%%%%%%%%%%%%%%%%%%%%%%%%%%%%%%%%%%%%%%%%%%%%%%%%%%%%%%%%%%%%%%%%%%%%%%%%%%%%%%%%%%
%%%%%%%%%%%%%%%%%%%%%%%%%%%%%%%%%%%%%%%%%%%%%%%%%%%%%%%%%%%%%%%%%%%%%%%%%%%%%%%%%%%%%%%%%%%%%%%%%%%%%%%%%%%%%%%%%%%%%%%%%%%%%%%%%%%%%%%%%%%%%%%%%%%%%%%%%%%%%
\begin{thm}{\label{Th2}} Let $(R,\fm, k)$ be a local ring, $C$ be a semidualizing $R$-complex, $X\in\mathrm{D}_{\Box}^f(R)$ and that $\id_R(C\utp_R X)<\infty$. Then, for all integers $t$, there is an equality $$\beta^{R}_{t}(X)=\underset{\tiny{i+j=t}}\sum\mu_{R}^i(\fm,C)\mu_{R}^{-j}(\fm,C\utp_R X).$$
\begin{proof}By \cite[Proposition 4.4]{CF1} and \cite[4.6(a)]{CF1}, we have $X\in\mathcal{A}_C(R)$, and so the canonical map $\gamma_X^C: X\rightarrow\uhom_R(C,C\utp_R X)$ is an isomorphism. Therefore, in view of \cite[Theorem 4.4.6]{CHFox}, we have the following equalities

\[\begin{array}{rl}
\beta^{R}_{t}(X)&=\vdim_k\H_{t}(k\utp_R X)\\
&=\vdim_k\H_{t}(k\utp_R \uhom_R(C,C\utp_R X))\\
&=\vdim_k\H_{t}(\uhom_{R}(\uhom_{R}(k,C),C\utp_R X))\\
&=\vdim_k\H_{t}(\uhom_{R}(\uhom_{R}(k,C)\utp_{k}k,C\utp_R X))\\
&=\vdim_k\H_{t}(\uhom_{k}(\uhom_{R}(k,C),\uhom_{R}(k,C\utp_R X)))\\
&=\underset{\tiny{i\in\mathbb{Z}}}\sum\vdim_k\H_{-i}(\uhom_{R}(k,C))\vdim_k\H_{-i+t}(\uhom_{R}(k,C\utp_R X))\\
&=\underset{\tiny{i+j=t}}\sum\mu_{R}^i(\fm, C)\mu_{R}^{-j}(\fm,C\utp_R X).
\end{array}\] This finishes the proof.
\end{proof}
\end{thm}

%%%%%%%%%%%%%%%%%%%%%%%%%%%%%%%%%%%%%%%%%%%%%%%%%%%%%%%%%%%%%%%%%%%%%%%%%%%%%%%%%%%%%%%%%%%%%%%%%%%%%%%%%%%%%%%%%%%%%%%%%%%%%%%%%%%%%%%%%%%%%%%%%%%%%%%%%%%%%%%%%%%%
\begin{cor}Let $(R,\fm)$ be a $d$-dimensional local ring, $C$ be a semidualizing $R$-module and $M$ be a finitely generated $R$-module such that $\Cid_R M<\infty$. Then, there exists an equality $$\nu(M)=\mu_{R}^d(\fm, C)\mu_{R}^{d}(\fm,C\otimes_R M).$$ In particular, if $\nu (M)=1$, then $C$ is the canonical module.
\begin{proof}By \cite[Theorem 2.11]{TW}, we have $\Cid_RM=\id_R(C\otimes_RM)$. Hence $R$ is Cohen-Macaulay by the Theorem of Bass, and so $C$ is maximal Cohen-Macaulay. Hence, $\mu_{R}^d(\fm, C)=0$ for all $i<d$. Also, $\mu_{R}^{d}(\fm,C\otimes_R M)=0$ for all $i>d$. On the other hand, by \cite[Theorem 2.11]{TW} and \cite[Corollary 2.9]{TW} we have $\Tor^{R}_{i}(C,M)=0$ for all $i\geq1$. Thus $C\utp_R M\simeq C\otimes_R M$, and the desired equality follows from Theorem \ref{Th2}. Finally, if $\nu(M)=1$, then one has $\mu_{R}^d(\fm, C)=1$ which shows that $C$ is the canonical $R$-module.
\end{proof}
\end{cor}

%%%%%%%%%%%%%%%%%%%%%%%%%%%%%%%%%%%%%%%%%%%%%%%%%%%%%%%%
\end{document}